\documentclass[12pt]{amsart}

\usepackage{amsmath}
\usepackage{amsfonts}
\usepackage{amssymb}
\usepackage{amsthm}
\usepackage{mathtools}
\usepackage{enumerate}
\usepackage{mathtools}

\usepackage[all]{xy}
\CompileMatrices
\usepackage{pb-diagram}
\usepackage{pb-xy}
\usepackage{bm}

\usepackage{color}

\newtheorem{theorem}{Theorem}[section]
\newtheorem{lemma}[theorem]{Lemma}
\newtheorem{proposition}[theorem]{Proposition}
\newtheorem{corollary}[theorem]{Corollary}
\newtheorem{remark}[theorem]{Remark}

\newtheorem{example}[theorem]{Example}
\newtheorem{examples}[theorem]{Examples}

\def\zcl{\protect\operatorname{zcl}}
\def\TC{\protect\operatorname{TC}}
\def\cat{\protect\operatorname{cat}}

\def\Imm{\mathrm{Imm}}

\title[Motion planning in real flag manifolds]{Motion planning in real flag manifolds\textsuperscript{\,$\natural$}}
\author{Jes\'us Gonz\'alez\textsuperscript{\dag}}
\author{B\'arbara Guti\'errez}
\author{Darwin Guti\'errez}
\author{Adriana Lara\textsuperscript{\ddag}}
\thanks{\textsuperscript{$\natural$}~~This paper is part of the Ph.D.~thesis work of the third author}
\thanks{\textsuperscript{\dag}~~Supported by Conacyt Research Grant 221221.}
\thanks{\textsuperscript{\ddag}~~Supported by Grant SIP20152082.}

\address{Departamento de Matem\'aticas, \newline \indent Centro de Investigaci\'on y de Estudios Avanzados del IPN, \newline \indent Av.~IPN 2508, Zacatenco, Mexico City 07000, Mexico.}
\email{jesus@math.cinvestav.mx}
\email{bgutierrez@math.cinvestav.mx}

\address{Departamento de Formaci\'on B\'asica, \newline \indent Escuela Superior de C\'omputo del IPN, \newline \indent Juan de Dios B\'atiz esq.~Miguel Oth\'on de Mendiz\'abal, M\'exico City 07738, Mexico.}
\email{dargut@hotmail.com}

\address{Departamento de Matem\'aticas, \newline \indent Escuela Superior de F\'isica y Matem\'aticas del IPN, \newline \indent Edificio 9, U.P.~Adolfo L\'opez Mateos, Mexico City 07300, Mexico.}
\email{adriana@esfm.ipn.mx}

\begin{document}
\begin{abstract}
Starting from Borel's description of the mod-2 cohomology of real flag manifolds, we give a minimal presentation of the cohomology ring for semi complete flag manifolds $F_{k,m}:=F(1,\ldots,1,m)$ where $1$ is repeated $k$ times. The information is used in order to estimate Farber's topological complexity of these spaces when $m$ approaches (from below) a 2-power. In particular, we get almost sharp estimates for $F_{2,2^e-1}$ which resemble the known situation for the real projective spaces $F_{1,2^e}$. Our results indicate that the agreement between the topological complexity and the immersion dimension of real projective spaces no longer holds for other flag manifolds. More interestingly, we also get corresponding results for the $s$-th (higher) topological complexity of these spaces. Actually, we prove the surprising fact that, as $s$ increases, the estimates become stronger. Indeed, we get several full computations of the higher motion planning problem of these manifolds. This property is also shown to hold for surfaces: we get a complete computation of the higher topological complexity of all closed surfaces (orientable or not). A homotopy-obstruction-theory explanation is included for the phenomenon of having a cohomologically accessible higher topological complexity even when the regular topological complexity is not so accessible.
\end{abstract}
\maketitle\tableofcontents

\section{Introduction and main results}
The concept of topological complexity (TC) of a space $X$ was introduced early this millennium by Michael Farber as a way to utilize techniques from homotopy theory in order to model and study, from a topological perspective, the motion planning problem in robotics. If $P(X)$ stands for the space of free paths in $X$, then $\TC(X)$ is the reduced Schwarz genus (also known as sectional category) of the fibration $e\colon P(X)\to X\times X$ given by $e(\gamma)=(\gamma(0),\gamma(1))$. We refer the reader to the book~\cite{MR2455573} and the references therein for a discussion of the meaning, relevance, and basic properties of Farber's concept. 

\medskip The idea was generalized a few years latter by Yuli Rudyak, who defined in~\cite{MR2593704} the $s$-th topological complexity of $X$, $\TC_s(X)$, as the reduced Schwarz genus of the $s$-th fold evaluation map $e_s\colon P(X)\to X^s$ given by $$e_s(\gamma)=\left(\gamma(0),\gamma\left(\frac{1}{s-1}\right),\gamma\left(\frac{2}{s-1}\right),\ldots,\gamma\left(\frac{s-2}{s-1}\right),\gamma(1)\right).$$
In particular $\TC=\TC_2$. Rudyak's ``higher'' topological complexity has been studied systematically in~\cite{MR3331610}.

\medskip
The purpose of this paper is two fold. For one, we give extensive computations to estimate the value of  $\TC_s$ on a number of infinite families of semicomplete real flag manifolds $F(1^k,m)$---the Grassmann type manifolds consisting of $k+1$ tuples $(L_1,\ldots L_k,V)$ of mutually orthogonal linear subspaces of $\mathbb{R}^{m+k}$ with $\dim(V)=m$ and $\dim(L_i)=1$ for $1\leq i\leq k$. Much of the motivation here comes from an amazing and unexpected connection between Farber's TC and one of the central problems in differential topology, namely the Euclidean immersion dimension for  smooth manifolds. Explicitly, for a manifold $M$, let $\mathrm{Imm}(M)$ denote the dimension of the smallest Euclidean space where $M$ can be immersed. Then the main result in~\cite{MR1988783} asserts that $\TC=\mathrm{Imm}$ for all real projective spaces $\mathbb{R}\mathrm{P}^m$ except for the only three parallelizable manifolds, $\mathbb{R}\mathrm{P}^1$, $\mathbb{R}\mathrm{P}^3$, and $\mathbb{R}\mathrm{P}^7$, for which the relation $\TC=\Imm-1$ holds. Of course, flag manifolds $F(1^k,m)$ are a natural generalization of real projective spaces. So  it is natural to ask whether the above relationship between topological complexty and immersion dimension also holds for the larger family of manifolds. Although the Euclidean immersion dimension of real flag manifolds is a much studied problem, and quite a lot of numeric information on it is available to date, the topological complexity of real flag manifolds had not been considered before---except, of course, for the already noted results with real projective spaces. The numeric $\TC$-results in this paper now show that the nice relationship between TC and Imm holding for real projective spaces $F(1,m)$ does not hold for flag manifolds $F(1^k,m)$ with $k>1$. For instance, $F(1,1,1)$ is a closed parallelizable 3-manifold~(\cite{MR809636}), so that $\mathrm{Imm}(F(1,1,1))=4$. However, Theorem~\ref{tcf13mintro} below gives $\TC(F(1,1,1))\in\{5,6\}$. Thus the relation `$\TC=\mathrm{Imm}-1$' holding for parallelizable real projective spaces no longer holds in the case of the other parallelizable flag manifolds $F(1^k,1)$. In general, flag manifolds (whether parallelizable of not) seem to have a larger TC than an Imm. For instance, the 7-dimensional flag manifold $F(1,1,3)$ has $\mathrm{Imm(F(1,1,3))}=10$~(\cite{MR0431194,MR702304}) whereas, according to Theorem~\ref{tcf13mintro} below, $\TC(F(1,1,3))\in\{13,14\}$.

\medskip
The second purpose of this paper aims at exhibiting subtle but substantial differences between Farber's original concept and Rudyak's extended definition. The point is that, after an initial examination, Rudyak's higher TC could seem to be a close relative of Farber's TC. For instance, it has been shown that the families of spaces $X$ whose TC has been computed have an equally computable higher TC (cf.~\cite{MR3331610,MR3373948,GGY}). Likewise, some theoretical results for TC have reasonable (although sometimes more complicated to prove) higher TC generalizations, see for instance~\cite{carrasquel,MR3396996,MR3117387,MR3020869}. However, other interesting theoretical properties known for TC do not have a known higher TC counterpart. For instance, it is known that the standard upper bound $2\dim(X)$ for $\TC(X)$ can be lowered by one unit whenever $\pi_1(X)=\mathbb{Z}_2$~(\cite{MR2649230}). Before this paper, it was not even clear whether the proof of such a fact could be generalized to the higher TC realm. As a consequence of Theorem~\ref{familias} below (with $k=1$), we now have that such a potential $\TC_s$ generalization is doomed to fail for $s\geq3$.

\medskip
Closely related to the above fact is the phenomenon that there are infinite families of spaces for which the computation of their TC would require a non-elementary homotopy theoretic argument, but whose higher TC can be computed using purely cohomological (i.e.~much simpler) methods. Indeed, as shown in this paper (Theorems~\ref{familias} and~\ref{teosur}), surfaces and flag manifolds $F(1^k,2^e-k+1)$ 
with $k\leq3$ have such a property\footnote{It might be the case that the restriction on $k$ can be removed---see the second half of Remark~\ref{dosnotasaclaras}.}. Technically speaking, this phenomenon can be summarized by saying that, in many cases, {\em the $\TC_2$-obstruction described in~\cite[Theorem~7]{MR2649230} vanishes without the vanishing of the analogous $\TC_s$-obstruction for $s\geq3$} (see the comments following~(\ref{comparacion1}) below).

\medskip
We next state our main results and explain how they fit within the introductory considerations above. Further comments will be given throughout the paper.

\begin{theorem}[Corollary~\ref{tcf13m}]
\label{tcf13mintro}
Let $k$ and $m$ be positive integers, $\delta\in\{0,1,\ldots,k-1\}$, and set $\epsilon=\min(\delta,1)$ and $\alpha(r)=\max(0,r)$. If $e$ is a nonnegative integer satisfying $2\delta\leq2^e\leq m+\delta$, then
\begin{equation}\label{estimate357}
(k-\delta+\epsilon)(2^{e+1}-1)+\hspace{.2mm}\alpha\hspace{-.5mm}\left({(\delta-1)(2^e-1)}\right)-\epsilon\leq\TC(F(1^k,m))\leq k(2m+k-1).
\end{equation}
\end{theorem}

Of course, the parameter $e$ should be taken as large as possible in order to get the full strength of Theorem~\ref{tcf13mintro}. Two special cases (where the estimate in~(\ref{estimate357}) has a gap of a unit) should be singled out from this result, namely: 
\begin{itemize}
\item $\TC(F(1,2^e))\in\{2^{e+1}-1,2^{e+1}\}$.
\item $\TC(F(1^2,2^e-1))\in\{2^{e+2}-3,2^{e+2}-2\}$.
\end{itemize}
Since $F(1,2^e)$ is the real projective space $\mathbb{R}\mathrm{P}^{2^e}$, the first situation is resolved by the well known equality
\begin{equation}\label{comparacion1}
\TC(\mathbb{R}\mathrm{P}^{2^e})=2^{e+1}-1
\end{equation}
(see~\cite{MR1988783}). It might seem reasonable to expect  $\TC(F(1^2,2^e-1))=2^{e+2}-3$. In any case, proving (disproving) such an equality is equivalent to showing the triviality (non-triviality) of the homotopy obstruction described in~\cite[Theorem~7]{MR2649230} for $X=F(1^2,2^e-1)$. The relevance of such a task becomes apparent by noticing that, as a special case of Theorem~\ref{familias} below, the $\TC_s$-analogue of the above homotopy obstruction does not vanish for these spaces when $s\geq3$:

\begin{theorem}[Theorems~\ref{familiasbis}]\label{familias}
For positive integers $e$, $k$ and $s$ with $e\geq1+\lfloor\frac{k-1}{2}\rfloor$ 
and $k\leq3\leq s$, $\TC_s(F(1^k,2^e-k+1))=s\dim(F(1^k,2^e-k+1))$.
\end{theorem}

For instance, when $k=1$, we get $\TC_s(\mathbb{R}\mathrm{P}^{2^e})= 2^e s$ if $s\geq3$ and $e\geq1$, which is certainly {\em not the case} for $s=2$, as noted in~(\ref{comparacion1}). The restriction $e\geq1$ is also needed as $\TC_s(S^1)=s-1$ is well known~(\cite[Section~4]{MR2593704}).

\medskip
It should be noted that Theorem~\ref{familias} can be thought of as a very distinguished manifestation of a more general phenomenon, namely: for fixed $k$ and $m$, the mod-2 cohomological estimates in this paper for $\TC_s(F(1^k,m))$ become sharper as $s$ increases. Such a point will be clarified and worked out in Section~\ref{sechigherTC} of this paper (Remark~\ref{aclaracion} and Corollary~\ref{monotonicidad}). 

\medskip
Other interesting (almost-sharp) estimates for the higher topological complexity of some semi-complete flag manifolds not considered in Theorem~\ref{familias} are discussed in Section~\ref{sechigherTC}. All together, or results seem to point out to what could be the best estimate that purely cohomological methods can yield for the higher topological complexity of semi complete flag manifolds $F(1^k,m)$ (see Remark~\ref{conclusion}).

\medskip
Theorems~\ref{tcf13mintro} and~\ref{familias} (and related results discussed in Section~\ref{sechigherTC}) are based on the identification of suitably long products of zero-divisors. The form of the required factors follows patterns that depend strongly on the value of $k$. The identification of such patterns is a major task in this paper that has greatly benefitted from the help of extensive computer calculations. On the other hand, the complexity of the calculations supporting Theorems~\ref{tcf13mintro} and~\ref{familias} is in sharp contrast with the easy situation for the complex analogues $F_{\mathbb{C}}(n_1,\ldots,n_\ell)$. The latter manifolds are 1-connected and symplectic (even K\"ahler), so their $s$-th topological complexity is well known (and easy to see) to agree with $s(\dim(F_{\mathbb{C}}(n_1,\ldots,n_\ell)))/2$ (see~\cite[Corollary~3.15]{MR3331610}). 

\medskip
The final section of the paper deals with the proof of:

\begin{theorem}[Proposition~\ref{teosurbis}]\label{teosur}
Let $S$ be a closed surface (orientable or not) other than the sphere and the torus. Then $\TC_s(S)=2s$ provided $s\geq3$.
\end{theorem}

It is folklore that the conclusion of Theorem~\ref{teosur} holds true for $s=2$ if $S$ is orientable and has genus at least 2. What is most interesting is to compare Theorem~\ref{teosur} with the fact that the precise value of Farber's topological complexity of non-orientable surfaces of genus at least 2 has become an intriguing open question in the field\footnote{There is an argument in the recent preprint~\cite{drani} for the equality $\TC_2(N_g)=4$ for $g\geq4$.}. 

\medskip
Theorem~\ref{teosur}, which will be a relatively easy consequence of the calculations supporting Theorem~\ref{familias}, gives of course an infinite family of spaces for which, just as for the flag manifolds in Theorem~\ref{familias}, the $\TC_s$ accessibility contrasts with the hardness of the $\TC_2$ situation. It would be interesting to know if such a phenomenon holds for other families of spaces.

\section{Cohomology and LS-category of $F(1^k,m)$}\label{cohomologiuadeunfactor}
Unless otherwise noted, all cohomology rings we deal with have $\mathbb{F}_2$-coefficients. For $i\geq0$, let $e_i$ denote the $i$-th elementary symmetric polynomial, and $h_i$ denote the $i$-th complete symmetric polynomial ($e_0=h_0=1$). In both cases the relevant variables will be explicitly indicated.

\begin{proposition}\label{f1km}
Let $m\geq1$. A minimal presentation for the ring $H^*(F(1^k,m))$ is given by generators $x_i$, $1\leq i\leq k$, all of dimension $1$, subject to the relations
\begin{equation}\label{mrf1km}
h_{m+i}(x_1,\ldots,x_{k+1-i})=0,\quad1\leq i\leq k.
\end{equation}
A graded additive basis for $H^*(F(1^k,m))$ is given by the monomials
\begin{equation}\label{basisf1km}
x(n_1,\ldots,n_k):=\prod_{i=1}^kx_i^{n_i}
\end{equation}
where $n_i\leq m+k-i$, for $i=1,\ldots,k$.
\end{proposition}

\begin{remark}\label{hausmann}{\em
The above presentation is a strong generalization of the one given in~\cite[Example~9.5.17]{hausmann} for complete flags---the latter one is not minimal. The direct proof below should be compared to~\cite{MR3007918}, a paper devoted to the proof (using Gr\"obner bases) of Proposition~\ref{f1km}.
}\end{remark}

\begin{proof}[Proof of Proposition~\ref{f1km}]
For $i=1,\ldots,k+1$, let $\gamma_i$ stand for the $i$-th tautological bundle on $F(1^k,m)$, and set $x_i=w_1(\gamma_i)$ for $i\leq k$, and $w_j=w_j(\gamma_{k+1})$ for $j\geq0$, the indicated Stiefel-Whitney classes. Borel's (non-minimal) presentation of $H^*(F(1^k,m))$ has generators $x_i$ and $w_j$ with the single (non-homogeneous) relation $$\sum_{j\geq0}w_j\prod_{i=1}^{k}(1+x_i)=1.$$ This expression's component in dimension $j>0$ is
\begin{equation}\label{reh}
\sum_{0\leq t\leq j}w_{j-t}e_t(x_1,\ldots,x_k)=0.
\end{equation}
In particular, for $j=1$, we get $w_1=e_1(x_1,\ldots,x_k)=h_1(x_1,\ldots,x_k)$. Assuming inductively that $w_\ell=h_\ell(x_1,\ldots,x_k)$ for $\ell<j$,~(\ref{reh}) gives $$w_j=\sum_{1\leq t\leq j}h_{j-t}(x_1,\ldots,x_k)e_t(x_1,\ldots,x_k)=h_j(x_1,\ldots,x_k).$$ This uses the basic relation between elementary and complete symmetric polynomials
\begin{equation}\label{basrelbeteleandcomsympols}
\sum_{t=0}^j(-1)^te_t(x_1,\ldots,x_k)h_{j-t}(x_1,\ldots,x_k)=0.
\end{equation}
Therefore, the generators $w_j$ are superfluous and, since $w_j=0$ for $j>m$, we get $h_{m+i}(x_1,\ldots,x_k)=0$ for $i>0$. This is~(\ref{mrf1km}) if $i=1$, otherwise use
\begin{equation}\label{descomponerinicio}
h_{m+i}(x_1,\ldots,x_k)=h_{m+i}(x_1,\ldots,x_{k-1})+x_kh_{m+i-1}(x_1,\ldots,x_k)
\end{equation}
to get $h_{m+i}(x_1,\ldots,x_{k-1})=0$ for $i>1$. Iteration of this argument yields~(\ref{mrf1km}). Further, these equations can be used to write any power $x_i^{\ell}$ with $\ell>m+k-i$ in terms of powers $x_j^n$ with $j<i$ or $n<\ell$. This shows that the monomials in~(\ref{basisf1km}) are additive generators of $H^*(F(1^k,m))$. 

On the other hand, the inclusion of the fiber in the total space of the fibration $F(1^{k-1},m)\to F(1^k,m)\to\mathbb{R}P^{m+k-1}$ is surjective in mod~2 cohomology. Therefore the corresponding $\mathbb{F}_2$-Serre spectral sequence has trivial coefficients and collapses from its second stage (cf.~Theorem~4.4 in page 126 of~\cite[Part~I]{MR1122592}). An easy inductive argument\footnote{Alternatively see~\cite[Corollary~9.5.15]{hausmann}.} then shows that the $\mathbb{F}_2$-Poincar\'e polynomial of $F(1^k,m)$ is $$P(x)=\prod_{i=1}^k\frac{\;\;1-x^{i+m\;}}{\!\!\!\!\!\!1-x}.$$ The proof is complete since $P(1)=\prod_{i=1}^k(m+i)$, which is the number of monomials in~(\ref{basisf1km}).
\end{proof}

The relations in~(\ref{mrf1km}) are a distilled form of a more general (equivalent but non-minimal) set of relations: our proof gives in fact
\begin{equation}\label{opgif}
h_{m+i}(x_1,\ldots,x_{k-j})=0 \quad\mbox{if}\quad i>j\geq0.
\end{equation}
(Alternatively,~(\ref{opgif}) is a consequence of~(\ref{mrf1km}) and the obvious inclusions $F(1^k,m)\hookrightarrow F(1^k,m+1)\hookrightarrow F(1^k,m+2) \hookrightarrow\cdots$.) In addition, the obvious action of the symmetric group $\Sigma_k$ on (the cohomology of) $F(1^k,m)$ implies that the relations in~(\ref{opgif}) extend to
\begin{equation}\label{relacionesextendidas}
h_{m+i}(x_{\ell_1},\ldots,x_{\ell_{k-j}})=0
\end{equation}
for any $1\leq\ell_1<\cdots<\ell_{k-j}\leq k$ with $0\leq j<i\leq k$. For instance,
\begin{equation}\label{altuhigh}
x_i^{m+k}=0 \neq x_i^{m+k-1}\mbox{ \ for any \ }i=1,\ldots,k,
\end{equation} 
where the non-triviality of $x_i^{m+k-1}$ comes from~(\ref{basisf1km}). As noted in~\cite[Example~3.1]{MR3007918}, this recovers the calculation in~\cite{MR2029922} of the heights of the generators $x_i$'s. Proposition~\ref{f1km} also allows us to recover the calculation of $\cat(F(1^k,m))$ in~\cite{MR2029922} (we use the normalized version of the Lusternik-Schnirelmann category, so that a contractible space $X$ has $\cat(X)=0$):

\begin{corollary}\label{lscatf1km}
$\cat(F(1^k,m))=\dim(F(1^k,m))=km+k(k-1)/2$.
\end{corollary}
\begin{proof}
It is well known that $km+k(k-1)/2=\dim(F(1^k,m))\geq\cat(F(1^k,m))$. The latter term is bounded from below by the $\mathbb{F}_2$-cup-length of $F(1^k,m)$ which, in view of Proposition~\ref{f1km}, is no less than $km+k(k-1)/2$ since  $x_1^{m+k-1}x_2^{m+k-2}\cdots x_k^{m}\neq0$.
\end{proof}

\begin{corollary}\label{kill1}
The annihilator of the (non-trivial) class $$x_1^{m+k-1}x_2^{m+k-2}\cdots x_k^m\in H^*(F(1^k,m))$$ is the maximal ideal $H^{{}>0}(F(1^k,m))$ of positive-degree elements. More precisely, $$x_1^{m+k-1}x_2^{m+k-2}\cdots x_j^{m+k-j}x_j=0$$ for $1\leq j \leq k$.
\end{corollary}
\begin{proof}
Apply, inductively on $j$, the relation~(\ref{mrf1km}) with $i=k-j+1$.
\end{proof}

Earlier versions of this work used the following generalization of the relations in~(\ref{opgif}):
\begin{proposition}\label{handy}
For $1\leq i\leq k$ and $j,\ell\geq0$, set $\tau(i,\ell)=(x_1\cdots x_i)^\ell$.
In the ring $H^*(F(1^k,m))$ we have
\begin{equation}\label{relsgenzadas}
\tau(i,\ell) h_j(x_1,\ldots,x_i)=0
\end{equation}
provided $i+j+\ell>m+k$.
\end{proposition}
\begin{proof}
By induction on $i$. For $i=1$, $\tau(i,\ell) h_j(x_1,\ldots,x_i) =x_1^{\ell+j}$ which vanishes provided $\ell+j+1>m+k$ in view of~(\ref{mrf1km}). Assume the result is valid for a fixed $i$ with $1\leq i<k$, and assume in addition $(i+1)+j+\ell>m+k$, then
\begin{eqnarray*}
\lefteqn{\tau(i+1,\ell) h_j(x_1,\ldots, x_{i+1})=}\\
& = & \tau(i,\ell)x_{i+1}^\ell\left(x_{i+1}^j+x_{i+1}^{j-1}h_1(x_1,\ldots,x_i)+\cdots+ h_j(x_1,\ldots,x_i)\right)\\&=&\tau(i,\ell)\left(x_{i+1}^{\ell+j}+x_{i+1}^{\ell+j-1}h_1(x_1,\ldots x_i)+\cdots+x_{i+1}^\ell h_j(x_1,\ldots x_i)\right)\\
& = & \tau(i,\ell)\hspace{.3mm}h_{\ell+j}(x_1,\ldots,x_{i+1}).
\end{eqnarray*}
The last equality uses the inductive relation $\tau(i,\ell)h_{j+p}(x_1,\ldots, x_i)=0$ when $p\geq1$. The result then follows since $h_{\ell+j}(x_1,\ldots,x_{i+1})=0$ in view of~(\ref{opgif}).
\end{proof}

For example, we have $(x_1\cdots x_i)^{\ell}=0$ for $1\leq i\leq k$ and $i+\ell>m+k$.

\medskip
More important for our later purposes is the fact that, as indicated in the proof of Proposition~\ref{f1km}, the extended relations in~(\ref{opgif}) can be used in an inductive way to write any polynomial in the $x_i$'s in terms of the basis~(\ref{basisf1km}). We next show that the resulting process can be written down with a nice closed formula if certain basis elements are to be neglected.

\begin{proposition}\label{neglect}
Let $0\leq j\leq i\leq k$ with $i\geq 1$. In terms of the basis~$(\ref{basisf1km})$, all basis elements $x(n_1,\ldots,n_k)$ appearing in the expression of
$$
x_i^{m+k-j}+x_i^{m+k-i} e_{i-j}(x_1,\ldots x_{i-1})\in H^*(F(1^k,m))
$$
have $n_i<m+k-i$ and $n_\ell=0$ for $\ell>i$.
\end{proposition}
\begin{proof}
The cases $j=0$ and $j=i$ hold vacuously true in view of~(\ref{altuhigh}). The case $j=i-1$ follows by observing that a repeated use of~(\ref{descomponerinicio}) allows us to write the relations in~(\ref{mrf1km}) as
\begin{equation}\label{inicioinduction}
x_i^{m+k-i+1}=x_i^{m+k-i} h_1+x_i^{m+k-i-1} h_2+
\cdots+h_{m+k-i+1}
\end{equation}
where the complete symmetric polynomials are evaluated at the variables $x_1,\ldots, x_{i-1}$. All other cases ($0<j<i-1$) follow from an obvious (decreasing) inductive calculation using~(\ref{basrelbeteleandcomsympols}) and the corresponding analogue of~(\ref{inicioinduction}).
\end{proof}

\section{$\mathbb{F}_2$-zcl bounds for $\TC(F(1^k,m))$}\label{seczcl1k}
Most of the existing methods to estimate the topological complexity of a given space are cohomological in nature and are based on some form of obstruction theory. One of the most successful methods to estimate Farber's topological complexity is: 

\begin{proposition}\label{cotas}
Let $X$ have the homotopy type of an ($e-1$)-connected CW complex of dimension $d$. Then $\zcl_R(X)\leq\TC(X)\leq\frac{2d}{e}$.
\end{proposition}

For a proof see~\cite[Theorems~4 and~7]{MR1957228}. Here $R$ is a commutative ring with unit and, if $\Delta\colon X \hookrightarrow X\times X$ denotes the diagonal inclusion, then $\zcl_R(X)$ stands for the maximal number of elements in ker($\Delta^*\colon H^*(X\times X;R)\to H^*(X;R)$) having a non-trivial product\footnote{Twisted cohomology coefficients can be used, but the present setting suffices for our goals.}. We will only be concerned with $R=\mathbb{Z}_2$, and will omit reference of these coefficients while writing cohomology groups. Thus, $\Delta^*\colon H^*(X)\otimes H^*(X)\to H^*(X)$ is given by cup-multiplication, which explains the notation ``zcl" (zero-divisors cup-length) for elements in the kernel of $\Delta^*$.

We use the notation $\lambda_i$ (resp.~$\rho_i$) for the generators $x_i$ on the left (resp.~right) tensor factor of 
$H^*(F(1^k,m)\times F(1^k,m))=H^*(F(1^k,m))\otimes H^*(F(1^k,m))$. The sum $\lambda_i+\rho_i$, which is a zero-divisor, will be denoted by $z_i$.

\begin{lemma}\label{2alae}
In the ring $H^*(F(1^k,2^e))^{\otimes 2}$ we have
\begin{equation}\label{elnocero1}
(z_1\cdots z_k)^{2^{e+1}-1}\neq0.
\end{equation}
\end{lemma}

\begin{remark}\label{nowaved}{\em
When $k\leq2^e$,~(\ref{elnocero1}) is sharp in the sense that $z_j^{2^{e+1}}=\lambda_j^{2^{e+1}}+\rho_j^{2^{e+1}}=0$ for $1\le j\le k$, in view of~(\ref{altuhigh}). However, such an optimality in~(\ref{elnocero1}) is far from holding when $k>2^e$. For instance,~(\ref{elnocero1}) asserts that $z_1z_2z_3\neq0$ in $H^*(F(1^4))^{\otimes2}$, but we will show in fact (Proposition~\ref{2alae2alae} below) that $z_1^3z_2^3z_3^2\neq0$ in $H^*(F(1^4))^{\otimes2}$. A similar phenomenon holds for $F(1^3,2)$---replacing the use of Proposition~\ref{2alae2alae} by Theorem~\ref{prdtsf13m} below (with $k=3$).
}\end{remark}

\begin{proof}[Proof of Lemma~\ref{2alae}]
We proceed by induction on $k$. The case for $k=1$ is elementary and well known---note that $F(1,2^e)$ is the real projective space $\mathbb{R}\mathrm{P}^{2^e}$. Assume the result is valid for $k$ and consider the fibration
\begin{equation}\label{simplificaciondeDarwin}
F(1^k,2^e)\stackrel{\iota}\to F(1^{k+1},2^e)\stackrel{\pi}\to F(1,2^e+k)=\mathbb{R}\mathrm{P}^{2^e+k}
\end{equation}
where $\pi(L_1,\ldots,L_{k+1},V)=(L_1,V\oplus\bigoplus_{2\le i\le k+1}L_i)$. Since $\iota$ is surjective in cohomology,~\cite[Theorem~4.4]{MR1122592} shows that the Serre spectral sequence for the term-wise cartesian square of~(\ref{simplificaciondeDarwin}) has a trivial system of coefficients, and collapses from its second term. The result follows since, by the inductive hypothesis, the left-hand side term in~(\ref{elnocero1}) is non-zero in the second stage of the spectral sequence.
\end{proof}

\begin{remark}\label{tesisrem1}{\em
In the Ph.D.~thesis work of the third author,~(\ref{elnocero1}) was originally proven (for $k\leq2^e$) by checking, through direct calculation (with Corollary~\ref{kill1} playing a key role), that the basis element 
$$
\lambda_1^{2^e-k}\lambda_2^{2^e-k+1}\cdots\lambda_k^{2^e-1}\rho_k^{2^e}\rho_{k-1}^{2^e+1}\cdots\rho_1^{2^e+k-1}
$$
appears with coefficient 1 in the expression of $(z_1\cdots z_k)^{2^{e+1}-1}$. In turn, the latter fact was suggested by extensive computer calculations.
}\end{remark}

Proposition~\ref{cotas}, Corollary~\ref{lscatf1km}, and Lemma~\ref{2alae} yield the estimate in Corollary~\ref{zcllower2e} below for the topological complexity of manifolds $F(1^k,m)$ admitting an equatorial inclusion $F(1^k,2^e)\hookrightarrow F(1^k,m)$ with $2^e\leq m$.

 \begin{corollary}\label{zcllower2e}
Let $e$ denote the integral part of $\log_2(m)$. Then
\begin{equation}\label{cotasTC2e-generalizado}
k(2^{e+1}-1)\leq\TC(F(1^k,m))\leq k(2m+k-1).
\end{equation}
\end{corollary}

Note that the smallest gap in~(\ref{cotasTC2e-generalizado}) is of $k^2$ units (for $m=2^e$). On the other hand, the lower bound in~(\ref{cotasTC2e-generalizado}) is optimal in general. For instance, the gap of a unit for $\TC(F(1,2^e))$ coming from Corollary~\ref{zcllower2e} (with $k=1$) is resolved by~(\ref{comparacion1}). This is of course compatible with the first assertion in Remark~\ref{nowaved}. But for $k>1$ there is room for improvements of the lower bound in~(\ref{cotasTC2e-generalizado}) by a zero-divisors cup-length analysis of an intermediate space $F(1^k,m')$ with $2^e<m'\leq m$ and $m'$ not a power of 2. For instance, since~(\ref{altuhigh}) yields $z_j^{2^{e+1}}=0$ in $H^*(F(1^2,2^{e+1}-2))^{\otimes2}$, the only possibility to improve the lower bound in~(\ref{cotasTC2e-generalizado}) for $k=2$ via zero-divisors cup-length considerations can come only through the analysis of the case for $F(1^2,2^{e+1}-1)$. In fact, we next give zero-divisors cup-length bounds for  $\TC(F(1^k,m))$ inherent to the case $m=2^e-1$. As observed at the end of Remark~\ref{otrodesharpness}, our argument will apply only for $k\geq2$ ---after all, the case of $F(1,2^e-1)$, the real projective space of dimension $2^e-1$, has been one of the most difficult situations studied over the years~(see for instance~\cite{MR2443106,MR0144355}).

\begin{proposition}\label{2alae2alae}
Assume $k\geq2$ and $e\geq1$. In $H^*(F(1^k,2^e-1))^{\otimes2}$ we have
\begin{equation}\label{mascompli}
(z_1\cdots z_{k-1})^{2^{e+1}-1}z_k^{2^{e+1}-2}\neq0.
\end{equation}
\end{proposition}

\begin{remark}\label{otrodesharpness}{\em 
When $k\le2^e+1$,~(\ref{mascompli}) is almost sharp in the sense that
$z_j^{2^{e+1}}=0$ for $1\le j\le k$, in view of~(\ref{altuhigh}). As illustrated in Section~\ref{completeflags}, such an optimality fails in general for $k>2^e+1$. Also worth noticing is that, in Remark~\ref{combcomplexities} below, we give evidence suggesting $(z_1\cdots z_k)^{2^{e+1}-1}=0$ in Propositon~\ref{2alae2alae}. On the other hand, note that~(\ref{mascompli}) certainly fails for $k=1$.
}\end{remark}

\begin{proof}[Proof of  Proposition~{\rm\ref{2alae2alae}}]
The inductive argument in the proof of Lemma~\ref{2alae}, now replacing~(\ref{simplificaciondeDarwin}) by the fibrations
\begin{equation}\label{diferenteperolomismo}
F(1^k,2^e-1)\to F(1^{k+1},2^e-1)\to F(1,2^e+k-1)=\mathbb{R}\mathrm{P}^{2^e+k-1},
\end{equation}
shows that the general case in Proposition~\ref{2alae2alae} follows inductively from the case $k=2$. On the other hand, for the latter case,~(\ref{altuhigh}) gives in $H^*(F(1,1,2^e-1))^{\otimes2}$
\begin{eqnarray*}
z_1^{2^{e+1}-1}z_2^{2^{e+1}-2}&=&(\lambda_1+\rho_1)^{2^{e+1}-1}(\lambda_2+\rho_2)^{2^{e+1}-2}\\
&=&\left(\lambda_1^{2^e}\rho_1^{2^e-1}+\lambda_1^{2^e-1}\rho_1^{2^e}\right)\left(\lambda_2^2+\rho_2^2\right)^{2^e-1}\\
&=&\left( \lambda_1^{2^e}\rho_1^{2^e-1}+\lambda_1^{2^e-1}\rho_1^{2^e} \right) \left( (\lambda_2^2)^{2^{e-1}}(\rho_2^2)^{2^{e-1}-1}+(\lambda_2^2)^{2^{e-1}-1}(\rho_2^2)^{2^{e-1}} \right)\\
&=&\left( \lambda_1^{2^e}\rho_1^{2^e-1}+\lambda_1^{2^e-1}\rho_1^{2^e} \right) \left( \lambda_2^{2^{e}}\rho_2^{2^{e}-2}+\lambda_2^{2^{e}-2}\rho_2^{2^e} \right).
\end{eqnarray*}
Further, if $\mu$ stands for either $\lambda$ or $\rho$, the relations~(\ref{mrf1km}) give $\mu_2^{2^e}=\mu_1^2 A_\mu+\mu_1\mu_2^{2^e-1}$. Thus
\begin{align*}
z_1^{2^{e+1}-1}&z_2^{2^{e+1}-2}\\
&=\left( \lambda_1^{2^e}\rho_1^{2^e-1}+\lambda_1^{2^e-1}\rho_1^{2^e} \right) \left( \lambda_2^{2^{e}}\rho_2^{2^{e}-2}+\lambda_2^{2^{e}-2}\rho_2^{2^e} \right)\\
&=\left( \lambda_1^{2^e}\rho_1^{2^e-1}+\lambda_1^{2^e-1}\rho_1^{2^e} \right) \left( \left( \lambda_1^2A_\lambda+\lambda_1\lambda_2^{2^e-1}\right)\rho_2^{2^{e}-2}+\lambda_2^{2^{e}-2}\left( \rho_1^2A_\rho+\rho_1\rho_2^{2^e-1}\right) \right)\\
&=\lambda_1^{2^e}\rho_1^{2^e-1}\cdot \lambda_2^{2^e-2}\rho_1\rho_2^{2^e-1}+\lambda_1^{2^e-1}\rho_1^{2^e}\cdot \lambda_1\lambda_2^{2^e-1}\rho_2^{2^e-2}\\
&=\lambda_1^{2^e}\lambda_2^{2^e-2}\rho_1^{2^e}\rho_2^{2^e-1}+\lambda_1^{2^e}\lambda_2^{2^e-1}\rho_1^{2^e}\rho_2^{2^e-2}.
\end{align*}
The result follows as the two monomials in the last expression are basis elements.
\end{proof}

\begin{remark}\label{tesisrem2}{\em
In the Ph.D.~thesis of the second author,~(\ref{mascompli}) was originally proven (for $k\leq2^e+1$) by checking, through direct calculation (with Corollary~\ref{kill1} and Proposition~\ref{handy} playing a key role), that the basis elements
$$
\beta_1=\left(\prod_{i=1}^{k-2}\lambda_i^{2^e-k+i}\right)\lambda_{k-1}^{2^e}\lambda_k^{2^e-2}\rho_k^{2^e-1}\left(\prod_{i=1}^{k-1}\rho_i^{2^e+k-i-1}\right)
$$
and
$$
\beta_2=\left(\prod_{i=1}^{k-2}\lambda_i^{2^e-k+i}\right)\lambda_{k-1}^{2^e}\lambda_k^{2^e-1}\rho_k^{2^e-2}\left(\prod_{i=1}^{k-1}\rho_i^{2^e+k-i-1}\right)
$$
appear with non-trivial coefficient in the expansion of the left-hand term in~(\ref{mascompli}). Once again, such a fact was suggested by extensive computer calculations.
}\end{remark}

\begin{remark}\label{combcomplexities}{\em
Before discussing the implications of Proposition~\ref{2alae2alae} to the topological complexity of flag manifolds, we make a brief pause to say a few words about the sharpness of Proposition~\ref{2alae2alae} when $2\leq k\leq2^e+1$ ---hypothesis that will be in force in this paragraph. Since $0=z_i^{2^{e+1}}\in H^*(F(1^k,2^e-1))^{\otimes 2}$ for all $i$, the triviality of any product $z_{i_1}\cdots z_{i_t}$ with $t\geq(2^{e+1}-1)k$ is equivalent to
\begin{equation}\label{eleanter}
\left(z_1\cdots z_k\right)^{2^{e+1}-1}=0.
\end{equation}
Proving~(\ref{eleanter}) presents a major challenge---not addressed in this work. Checking the validity of~(\ref{eleanter}) for $k=2$ is a simple matter in view of the last expression for $z_1^{2^{e+1}-1}z_2^{2^{e+1}-2}$ at the end of the proof of Proposition~\ref{2alae2alae}. We have checked the validity of~(\ref{eleanter}) for $k\in\{3,4\}$ with the help of a computer, but the task quickly becomes computationally prohibitive as the number of basis elements in the expression on the left-hand term of~(\ref{mascompli}) increases very fast as $k$ grows: 16 basis elements are need for $k=3$, while the number of required basis elements increases to $1128$ for $k=4$---the sum of which would have to vanish after multiplying by $z_k$, should~(\ref{eleanter}) be true.}
\end{remark}

 \begin{corollary}\label{zcllower2emenos1}
Let $e$ denote the integral part of $\log_2(m+1)$. If $e\geq1$ and $k\geq2$, then
\begin{equation}\label{cotasTC2emenos1-generalizado}
k(2^{e+1}-1)-1\leq\TC(F(1^k,m))\leq k(2m+k-1).
\end{equation}
\end{corollary}

\begin{remark}\label{atrayente}{\em
The smallest gap in~(\ref{cotasTC2emenos1-generalizado}) is of $k(k-2)+1$ units (for $m=2^e-1$). The case of $F(1^2,2^e-1)$ (e.g.~the closed parallelizable 3-manifold $F(1^3)$) is particularly appealing, as the corresponding gap is only of one unit. Also interesting to note is the obvious fibration $F(1,2^e-1)\to F(1,1,2^e-1)\to F(1,2^e)$ for which the TC of the base is well understood, the TC of the total space has been estimated in Corollary~\ref{zcllower2emenos1} with an error of at most a unit, and yet the search for the TC of the fiber (in the form of the immersion dimension of $\mathbb{R}\mathrm{P}^{2^e-1}$) has been one of the main driving forces shaping homotopy theory over the last 75 years.
}\end{remark}

\begin{example}\label{comportamientosimilar}{\em
It is elementary to see that, when $k=1$, Corollary~\ref{zcllower2e} captures all the $\TC$-information cohomologically available from Lemma~\ref{2alae}. Indeed, the relations $z_1^{2^{e+1}-1}\neq0= z_1^{2^{e+1}}$ holding in $H^*(F(1,2^e))^{\otimes2}$ clearly hold in any $H^*(F(1,m))^{\otimes2}$ with $2^e\le m<2^{e+1}$. Likewise, for $k=2$, Corollaries~\ref{zcllower2e} and~\ref{zcllower2emenos1} capture all the $\TC$-information available from Lemma~\ref{2alae} and Proposition~\ref{2alae2alae}. For, while $z_1^{2^{e+1}-1}z_2^{2^{e+1}-2}\neq0$ holds sharply in $H^*(F(1,1,2^e-1))^{\otimes2}$ (c.f.~Remark~\ref{combcomplexities}), $z_1^{2^{e+1}-1}z_2^{2^{e+1}-1}\neq0$ holds sharply in $H^*(F(1,1,m))^{\otimes2}$ for $2^e\leq m\leq2^{e+1}-2$ (c.f.~Remark~\ref{nowaved}).
}\end{example}

Example~\ref{comportamientosimilar} (and extensive computer calculations) seem to suggest that all the zcl-information for $F(1^k,m)$ is contained in the cases $m=2^e-\delta$ with $0\leq\delta<k$. The analysis of the corresponding zcl properties is the subject of the remainder of this section (see Theorem~\ref{prdtsf13m}).  

\medskip
Apply the inductive argument in the proofs of Lemma~\ref{2alae} and Proposition~\ref{2alae2alae}, this time with the fibration
$
F(1,2^e-2)\to F(1^k,2^e-2)\to F(1^{k-1},2^e-1).
$
For $k\geq3$ and $e\geq1$, Proposition~\ref{2alae2alae} gives $(z_1\cdots z_{k-2})^{2^{e+1}-1}z_{k-1}^{2^{e+1}-2}\neq0$ in $H^*(F(1^{k-1},2^e-1))^{\otimes2}$, whereas $z_1^{2^e-1}\neq0$ in $H^*(F(1,2^e-2))^{\otimes2}$ is a standard calculation if $e\geq2$. Thus
\begin{equation}\label{observaciondeenrique}
\left(z_1\cdots z_{k-2}\right)^{2^{e+1}-1}z_{k-1}^{2^{e+1}-2}z_k^{2^e-1}\neq0\mbox{\; in\;}H^*(F(1^k,2^e-2))^{\otimes2}\mbox{\;if\;}k\geq3\mbox{\;and\;} e\geq2.
\end{equation}
What is remarkable in~(\ref{observaciondeenrique}) is that, although this argument is really measuring the zero-divisors cup-lenght of some particular graded object associated to $H^*(F(1^k,2^e-2))^{\otimes2}$, extensive computer calculations suggest that no cohomological information has been missed.

\begin{remark}\label{1112e-2failure}{\em
At a first glance,~(\ref{comparacion1}) and Remark~\ref{atrayente} might suggest that the methods in this paper could lead to estimate the TC of $F(1^3,2^e-2)$ with an error of at most a unit. The error, however, increases exponentially with $e$ (see Example~\ref{properk3exe}). Yet, as shown in Section~\ref{sechigherTC} below, the corresponding $\TC_s$ estimates for $s\geq3$ will in fact be sharp. 
}\end{remark}

The argument leading to~(\ref{observaciondeenrique}) can be iterated with the fibrations
\begin{equation}\label{lasfibsiniciales}
F(1,2^e-\delta)\to F(1^k,2^e-\delta)\to F(1^{k-1},2^e-\delta+1)
\end{equation}
for $\delta\geq2$ (but note that the case $\delta=1$ fails to recover Proposition~\ref{2alae2alae}) to get the following generalizations of Lemma~\ref{2alae} and Proposition~\ref{2alae2alae}, and of Corollaries~\ref{zcllower2e} and~\ref{zcllower2emenos1}:

\begin{theorem}\label{prdtsf13m}
The following assertions hold in 
$H^*(F(1^k,m))^{\otimes2}\colon$
\begin{itemize}
\item[(a)] For $m+k\leq2^{e+1}$ and $1\leq i\leq k$, $z_i^{2^{e+1}}=0.$
\item[(b)] For $m=2^e-\delta$ with $k>\delta\geq0$ and $2^{e-1}\geq\delta$,
\begin{equation}\label{factoresgenerales}
\left(z_1\cdots z_{k-\delta}\right)^{2^{e+1}-1}z_{k-\delta+1}^{2^{e+1}-2}\left(z_{k-\delta+2}\cdots z_k\right)^{2^e-1}\neq0.
\end{equation}
\end{itemize}
\end{theorem}

\begin{corollary}\label{tcf13m}
Let $k$ and $m$ be positive integers, $\delta\in\{0,1,\ldots,k-1\}$, and set $\epsilon=\min(\delta,1)$ and $\alpha(r)=\max(0,r)$. If a nonnegative integer $e$ satisfies $2\delta\leq2^e\leq m+\delta$, then
\begin{equation}\label{gapk3}
(k-\delta+\epsilon)(2^{e+1}-1)+\hspace{.2mm}\alpha\hspace{-.5mm}\left({(\delta-1)(2^e-1)}\right)-\epsilon\leq\TC(F(1^k,m))\leq k(2m+k-1).
\end{equation}
\end{corollary}

Due to the form of the exponents of the factors on the left-hand side of~(\ref{factoresgenerales}), the gap in~(\ref{gapk3}) becomes in general larger as the parameter $\delta$ increases. Still, as shown in the following examples, there are concrete situations where Corollary~\ref{tcf13m} yields better lower bounds for larger values of $\delta$.

\begin{examples}\label{properk3exe}{\em
Obtaining the sharpest information from Theorem~\ref{prdtsf13m} and Corollary~\ref{tcf13m} for a fixed flag manifold $F(1^k,m_0)$ usually requires choosing a suitable combination of parameters $(e,\delta)$ with $2^e-\delta\leq m_0$ (so that the non-triviality of a cohomology class in $F(1^k,m_0)$ can be obtained, via Theorem~\ref{prdtsf13m}, from the non-triviality of its restriction to $F(1^k,2^e-\delta)$). Take for instance the case of $F(1^3,2)$ where the conclusion of item (b) in Theorem~\ref{prdtsf13m} with $\delta=0$ is $(z_1z_2z_3)^3\neq0$, but the conclusion with $\delta=2$ is in fact $z_1^7z_2^6z_3^3\neq0$. Alternatively, the case $\delta=0$ in Corollary~\ref{tcf13m} implies $\TC(F(1^3,6))\geq21$. However the case $\delta=2$ yields the stronger estimate $\TC(F(1^3,6))\geq36$. In particular, the smallest gap in~(\ref{gapk3}) for $F(1^3,6)$ is of 6 units and corresponds to $\delta=2$. More generally, for $e\geq2$, the smallest gap in~(\ref{gapk3}) for $F(1^3,2^e-2)$ is of $2^e-2$ units. In particular $\TC(F(1^3,2))\in\{16,17,18\}$---a gap of only two units.
}\end{examples}

As indicated in Remarks~\ref{nowaved} and~\ref{otrodesharpness}, the lower bound in~(\ref{gapk3}) tends to get weaker as $k$ is larger than $2^e-\delta$. The extreme case holds (with $(e,\delta)\in\{(0,0),(1,1)\}$) for complete flag manifolds. The next brief section illustrates the sort of phenomena found for those manifolds. 

\section{Examples with complete flag manifolds}\label{completeflags}
It is well known that complete flag manifolds $F_k:=F(1^k)$ (with $k$ ones) are parallelizable. In the case of $F_3$, the optimality of our cohomological methods has been discussed in Remark~\ref{combcomplexities} (the corresponding optimality in the case of $F_2\cong S^1$ is elementary). Likewise, in the case of $F_4$, Theorem~\ref{tcf13mintro} asserts $8\leq\TC(F_4)\leq12$, while computer calculations show that the lower bound $8\leq\TC(F_4)$ is all we can extract from $\zcl_{\mathbb{Z}_2}$ considerations. The case of $F_5$ is particularly interesting: although a straight application of Theorem~\ref{tcf13mintro} only gives $11\leq\TC(F_5)\leq20$, our cohomological \emph{methods} give in fact
\begin{equation}\label{gapdedos}
18\leq\TC(F_5)\leq 20.
\end{equation}
Indeed, since $z_1^7z_2^6\neq0$ in $H^*(F(1^2,3))^{\otimes2}$ and $z_1^3z_2^2\neq0$  in $H^*(F(1^3))^{\otimes2}$, the spectral sequence argument with the fibration $F(1^3)\to F_5\to F(1^2,3)$ shows
\begin{equation}\label{evidentia}
0\neq z_1^7z_2^6z_3^3z_4^2\in H^*(F_5)^{\otimes2}.
\end{equation}
Furthermore, computer calculations show that the lower bound in~(\ref{gapdedos}) is all we can extract for $F_5$ from $\zcl_{\mathbb{Z}_2}$ arguments. Similar considerations show
\begin{equation}\label{gapdedosbis}
25\leq\TC(F_6)\leq30
\end{equation}
although brute force (and extensive) computer calculations give in fact $27\leq\TC(F_6)$.


\section{Higher topological complexity}\label{sechigherTC}
We have seen that the methods in this paper give almost-sharp estimates for the topological complexity of flag manifolds $F(1,2^e)$ and $F(1^2,2^e-1)$. This section's goal is to show that, if we care about higher topological complexity, the estimates become sharp and, above all, valid for other flag manifolds of the form $F(1^k,2^e-k+1)$. In general, our results show that, as $s$ increases, the cohomological method becomes better suited to estimate $\TC_s(F(1^k,m))$. This point will be made precise in Remark~\ref{aclaracion} and Corollary~\ref{monotonicidad} below.

\medskip
The $\TC_s$-analogue of Proposition~\ref{cotas}, as noted in~\cite[Proposition~3.4]{MR2593704}, is stated in Proposition~\ref{cotasseq} below, where $\zcl_{s,R}(X)$, the $s$-th zero-divisors cup-length of $X$, stands for the maximal number of elements in ker($\Delta_s^*\colon H^*(X^s;R)\to H^*(X;R)$) having a non-trivial product, and $\Delta_s\colon X \to X^s$ is the iterated diagonal.

\begin{proposition}\label{cotasseq}
Let $X$ have the homotopy type of an ($e-1$)-connected CW complex of dimension $d$. Then
\begin{equation}\label{generalgap}
\zcl_{s,R}(X)\leq\TC_s(X)\leq\frac{sd}{e}.
\end{equation}
\end{proposition}

\begin{remark}\label{aclaracion}{\em
Let $G(k,m,s)$ denote the gap in~(\ref{generalgap}) for $R=\mathbb{Z}_2$ and $X=F(1^k,m)$, namely $G(k,m,s)=sd_{k,m}-\zcl_{s,\mathbb{Z}_2}(F(1^k,m))$ where $d_{k,m}=km+k(k-1)/2$ (c.f.~Corollary~\ref{lscatf1km}). Critically small values of $G(k,m,2)$ have been carefully pointed out in Section~\ref{seczcl1k}, in Examples~\ref{properk3exe}, as well as in~(\ref{gapdedos}) and~(\ref{gapdedosbis}). Corollary~\ref{monotonicidad} below indicates that, for $k$ and $m$ fixed, the sequence of non-negative integers $\{G(k,m,s)\}_{s\geq2}$ is monotonically decreasing---therefore eventually constant. In fact, in the main result of this section (Theorem~\ref{familiasbis} below), the monotonic phenomenon holds with a zero limiting value, $\lim_{s\mapsto \infty}G(k,m,s)=0$, thus getting sharp results.
}\end{remark}

\begin{theorem}\label{familiasbis}
For positive integers $e$, $k$ and $s$ with $e\geq1+\lfloor\frac{k-1}{2}\rfloor$ 
and $k\leq3\leq s$, $\zcl_{s,\mathbb{Z}_2}(F(1^k,2^e-k+1))=\TC_s(F(1^k,2^e-k+1))=s\dim(F(1^k,2^e-k+1))$.
\end{theorem}

Theorem~\ref{familiasbis} is an immediate consequence of Proposition~\ref{cotasseq} and the inequality $s\dim(F(1^k,2^e-k+1))\leq\zcl_{s,\mathbb{Z}_2}(F(1^k,2^e-k+1))$, which will be established in Propositions~\ref{prodsnocerok12} and~\ref{prodsnocerok3} for $k\leq3\leq s$ by identifying non-trivial products with suitably many $s$-th zero-divisors as factors. 

\smallskip
For $1\leq i\leq s$ and $1\leq j\leq k$, let $x_{i,j}$ be the pullback class $\pi_i^*(x_j)\in H^*(F(1^k,m))^{\otimes s}$ where $\pi_i\colon F(1^k,m)^s\to F(1^k,m)$ is the $i$-th projection ($1\leq i\leq s$ and $1\leq j\leq k$), and let $z_{i,j}$ stand for the $s$-th zero-divisor $x_{1,j}+x_{i,j}$. We will deal with the basis~(\ref{basisf1km}) and its tensor product  basis
\begin{equation}\label{basetensorial}
\prod_{i=1}^sx_i(n_{i,1},\ldots n_{i,k}), \quad 0\leq n_{i,j}\leq m+k-j,
\end{equation}
where $x_i(n_{i,1},\ldots n_{i,k})=\pi_i^*(x(n_{i,1},\ldots n_{i,k}))$.

\begin{proposition}\label{prodsnocerok12}
For $s\geq3$ and $e\geq1$, 
\begin{enumerate}
\item \label{paraadelante}$0\neq z_{2,1}^{2^{e+1}-1}\cdot z_{3,1}^{2^e+1}\cdot z_{4,1}^{2^e}\cdots z_{s,1}^{2^e}\in H^*(F(1,2^e))^{\otimes s}$.
\item $0\neq (z_{2,1}^{2^{e+1}-1} z_{2,2}^{2^{e+1}-2})\cdot (z_{3,1}^{2^e-1} z_{3,2}^{2^e+1})\cdot(z_{4,1}^{2^e}z_{4,2}^{2^e-1})\cdots(z_{s,1}^{2^e}z_{s,2}^{2^e-1})\in H^*(F(1^2,2^e-1))^{\otimes s}$.
\end{enumerate}
\end{proposition}

\begin{proposition}\label{prodsnocerok3}
For $s\geq3$ and $e\geq2$, $$0\neq(z_{2,1}^{2^{e+1}-1}z_{2,2}^{2^{e+1}-2}z_{2,3}^{2^e-1})
\cdot(z_{3,1}^{2^e-1}z_{3,2}^{2^e-1}z_{3,3}^{2^{e+1}-3})\cdot\prod_{i=4}^{s}(z_{i,1}^{2^e}z_{i,2}^{2^e-1}z_{i,3}^{2^e-2})\in H^*(F(1^3,2^e-2))^{\otimes s}.$$
\end{proposition}

\begin{remark}\label{dosnotasaclaras}{\em
Note that the powers of the factors $z_{2,j}$ in the three products above coincide with the relevant power(s) of the products in~(\ref{elnocero1}) with $k=1$,~(\ref{mascompli}) with $k=2$, and~(\ref{observaciondeenrique}) for $k=3$. In the present case ($s\geq3$), the form of the powers of the factors $z_{3,j}$ is what allows us to get sharp results. Also worth mentioning is the possibility that Theorem~\ref{familiasbis} could hold true by simplifying the restriction ``$k\leq3\leq s$'' to ``$k\leq s$'' (see Remark~\ref{conclusion} for a more general possibility). The proof of such an assertion seems to require computational input (suggesting suitable generalizations of Propositions~\ref{prodsnocerok12} and~\ref{prodsnocerok3}) that does not seem to be currently available with today's computer capabilities.
}\end{remark}

Lemma~\ref{stabilizacion} below implies that it suffices to prove Propositions~\ref{prodsnocerok12} and~\ref{prodsnocerok3} for $s=3$.
\begin{lemma}\label{stabilizacion}
For $2\leq i\leq s$, the expression of $$z_{i,1}^{m+k-1}z_{i,2}^{m+k-2}\cdots z_{i,k-1}^{m+1} z_{i,k}^{m}+x_i(m+k-1,m+k-2,\cdots,m+1,m)\in H^*(F(1^k,m))^{\otimes s}$$ in terms of the basis in~(\ref{basetensorial}) involves only basis elements of the form $$x_{1}(r_1,\ldots,r_k)\cdot x_i(t_1,\ldots,t_k)$$ with $r_j>0$ for some $j\in\{1,\ldots,k\}$ $\,($so $t_{j'}<m+k-j'$ for some $j'\in\{1,\ldots,k\})$.
\end{lemma}
\begin{proof}
Expand out $z_{i,1}^{m+k-1}z_{i,2}^{m+k-2}\cdots z_{i,k-1}^{m+1} z_{i,k}^{m}$ and notice that all the resulting monomials are basis elements.
\end{proof}

\begin{corollary}\label{monotonicidad}
For $k,m\geq1$, $G(k,m,2)\geq G(k,m,3)\geq G(k,m,4)\geq\cdots\geq0$.
\end{corollary}
\begin{proof}
If $z\in H^*(F(1^k,m))^{\otimes s}$ is some non-trivial product of $s$-th zero-divisors, Lemma~\ref{stabilizacion} implies that $z\cdot z_{s+1,1}^{m+k-1}z_{s+1,2}^{m+k-2}\cdots z_{s+1,k-1}^{m+1} z_{s+1,k}^{m}\in H^*(F(1^k,m))^{\otimes (s+1)}$ is non-trivial too. The result then follows from the bare definition of the function $G(k,m,s)$.
\end{proof}

Propositions~\ref{prodsnocerok12} and~\ref{prodsnocerok3} are proved by direct computation of the given products. In all cases, advantage is taken of the fact that the products lie in the top dimension $s(km+\binom{k}{2})$ of the relevant ring $H^*(F(1^k,m))^{\otimes s}$, where the additive basis~(\ref{basetensorial}) reduces to the single element
\begin{equation}\label{onlybasiselement}
\prod_{i=1}^{s}x_i(m+k-1,m+k-2,\ldots,m+1,m).
\end{equation}
Explicitly, we use the inductive process indicated in the proof of Proposition~\ref{f1km}, except that, since~(\ref{onlybasiselement}) is the only basis element we care about, the extended relations in~(\ref{opgif}) can  be replaced by the relations
\begin{equation}\label{noextendedreeplazo}
x_i^{m+k-j}=x_i^{m+k-i} e_{i-j}(x_1,\ldots x_{i-1}),\;\,\mbox{for} \,\;0\leq j\leq i\leq k \,\;\mbox{and}\,\; i\geq1,
\end{equation}
coming from Proposition~\ref{neglect}. Proof details for Proposition~\ref{prodsnocerok12} are similar (and easier) than those for Proposition~\ref{prodsnocerok3}, so we only focus on the latter case.

\begin{proof}[Proof of Proposition~\ref{prodsnocerok3}]
By Lemma~\ref{stabilizacion} (see also the proof of Corollary~\ref{monotonicidad}), we only need to consider the case $s=3$. We will show that for $e\geq2$,
\begin{equation}\label{amasar}
(z_{2,1}^{2^{e+1}-1}z_{2,2}^{2^{e+1}-2}z_{2,3}^{2^e-1})
\cdot(z_{3,1}^{2^e-1}z_{3,2}^{2^e-1}z_{3,3}^{2^{e+1}-3})=\prod_{i=1}^{3} x_{i,1}^{2^e}x_{i,2}^{2^e-1}x_{i,3}^{2^e-2},
\end{equation}
the top basis element in $ H^*(F(1^3,2^e-2))^{\otimes 3}$.

\smallskip
The usual mod-2 arithmetic of binomial coefficients, and the fact that $x_{i,j}^{2^e+1}=0$ give 
$$z_{3,3}^{2^{e+1}-3}=(x_{1,3}+x_{3,3})^{2^{e+1}-3}=x_{1,3}^{2^e}x_{3,3}^{2^e-3}+x_{1,3}^{2^e-3}x_{3,3}^{2^e}$$
(of course, this uses the hypothesis $e\geq2$). Due to the form of the relations~(\ref{noextendedreeplazo}) ---or~(\ref{opgif}) for that matter--- and since $x_{3,3}^{2^e-3}$ is a basis element, the term $x_{1,3}^{2^e}x_{3,3}^{2^e-3}$ above cannot contribute to the top basis element. In other words, the considerations around~(\ref{onlybasiselement}) imply that the product of the term $x_{1,3}^{2^e}x_{3,3}^{2^e-3}$ with the first five powers on the left of~(\ref{amasar}) vanishes. Such an argument will be used repeatedly in what follows, and will simply be referred to by using a ``$\equiv$'' symbol. In these terms, the relations~(\ref{noextendedreeplazo}) allow us to extend the short  calculation above to
$$z_{3,3}^{2^{e+1}-3}=(x_{1,3}+x_{3,3})^{2^{e+1}-3}=x_{1,3}^{2^e}x_{3,3}^{2^e-3}+x_{1,3}^{2^e-3}x_{3,3}^{2^e}\equiv x_{1,3}^{2^e-3}x_{3,3}^{2^e}=x_{1,3}^{2^e-3}x_{3,1}x_{3,2}\cdot x_{3,3}^{2^e-2},$$
\begin{eqnarray*}
z_{3,2}^{2^e-1}z_{3,3}^{2^{e+1}-3}&\equiv&z_{3,2}^{2^e-1}x_{1,3}^{2^e-3}x_{3,1}x_{3,2}\cdot x_{3,3}^{2^e-2}\\&=&(x_{1,2}+x_{3,2})^{2^e-1}x_{1,3}^{2^e-3}x_{3,1}x_{3,2}\cdot x_{3,3}^{2^e-2}\\
&\equiv&(x_{1,2}x_{3,2}^{2^e-2}+x_{3,2}^{2^e-1})x_{1,3}^{2^e-3}x_{3,1}x_{3,2}\cdot x_{3,3}^{2^e-2}\\ &=& x_{1,3}^{2^e-3}x_{3,1}(x_{1,2}+x_{3,1})\cdot x_{3,2}^{2^e-1}x_{3,3}^{2^e-2},
\end{eqnarray*}
and
\begin{eqnarray*}
z_{3,1}^{2^e-1}z_{3,2}^{2^e-1}z_{3,3}^{2^{e+1}-3}&\equiv&(x_{1,1}+x_{3,1})^{2^e-1}x_{1,3}^{2^e-3}x_{3,1}(x_{1,2}+x_{3,1})\cdot x_{3,2}^{2^e-1}x_{3,3}^{2^e-2}\\&\equiv&x_{1,3}^{2^e-3}(x_{1,1}x_{3,1}^{2^e-2}+x_{3,1}^{2^e-1})x_{3,1}(x_{1,2}+x_{3,1})\cdot x_{3,2}^{2^e-1}x_{3,3}^{2^e-2}\\&=&x_{1,3}^{2^e-3}(x_{1,1}+x_{1,2})\cdot x_{3,1}^{2^e}x_{3,2}^{2^e-1}x_{3,3}^{2^e-2}.
\end{eqnarray*}
An entirely similar (and straightforward) calculation gives
$$
z_{2,1}^{2^{e+1}-1}z_{2,2}^{2^{e+1}-2}z_{2,3}^{2^e-1}\equiv(x_{1,1}^{2^e}x_{1,2}^{2^e-2}x_{1,3}+x_{1,1}^{2^e-1}x_{1,2}^{2^e})\cdot x_{2,1}^{2^e}x_{2,2}^{2^e-1}x_{2,3}^{2^e-2},
$$
and the result then follows since an additional (and much simpler) such computation gives $x_{1,3}^{2^e-3}(x_{1,1}+x_{1,2})\cdot (x_{1,1}^{2^e}x_{1,2}^{2^e-2}x_{1,3}+x_{1,1}^{2^e-1}x_{1,2}^{2^e})\equiv x_{1,1}^{2^e}x_{1,2}^{2^e-1}x_{1,3}^{2^e-2}$.
\end{proof}

The Serre spectral sequence method used in Section~\ref{seczcl1k} could now be coupled with Propositions~\ref{prodsnocerok12} and~\ref{prodsnocerok3} to get an extension of Theorem~\ref{familiasbis} on the lines of Corollary~\ref{tcf13m}. However such a task would need to be done in a carefully selective way 
as, in some cases, the direct computations in the previous proof give better results. In fact, as the following example suggests (see also the proof of Proposition~\ref{nocero112}), best results can be obtained by a suitable combination of both techniques.

\begin{example}\label{sss112e}{\em
Proposition~\ref{prodsnocerok12}(\ref{paraadelante}) and the Serre spectral sequence applied to the fibration $F(1,2^e)\to F(1,1,2^e)\to F(1,2^e+1)$ (with $e\geq1$) yield the non-triviality of
$$
(z_{2,1}^{2^{e+1}-1}z_{2,2}^{2^{e+1}-1})\cdot(z_{3,1}^{2^e+1}z_{3,2}^{2^e+1})\cdot(z_{4,1}^{2^e}z_{4,2}^{2^e})\,\cdots\,(z_{s,1}^{2^e}z_{s,2}^{2^e})\in H^*(F(1,1,2^e))^{\otimes s}
$$
for $s\geq3$. But one cn do better. For instance, a direct argument (spelled out in Proposition~\ref{nocero112} below) gives in fact the non-triviality of
\begin{equation}\label{moreinvrec}
\left(z_{2,1}^3 z_{2,2}^3\right) \cdot\left( z_{3,1}^3 z_{3,2}^3\right)\cdot\left( z_{4,1}^3 z_{4,2}^2\right)\cdots \left(z_{s,1}^3 z_{s,2}^2\right)\in H^*(F(1,1,2))^{\otimes s}
\end{equation}
for $s\geq3$, so that $5s-3\leq\TC_s(F(1,1,2))\leq5s$. As a result we have that $G(2,2,s)\leq3$ provided $s\geq3$ (recall from Corollary~\ref{zcllower2e} that $G(2,2^e,2)=4$). In fact, extensive computer computations (not given here) suggest that 
\begin{equation}\label{decomparacion}
\mbox{$G(2,2,s)=3\;$ when $\,s\geq3$.}
\end{equation}
The key point then comes from the fact that Theorem~\ref{2alaehigher} below gives the sharper result\footnote{Computer calculations suggest that, in fact, $G(2,2^e,s)=1$ for $e\geq2$ and $s\geq3$.}
$$
\mbox{$G(2,2^e,s)\leq1\;$ for $s\geq3$ and $e\geq2$.}
$$
}\end{example}

\begin{proposition}\label{nocero112}
The element in~$(\ref{moreinvrec})$ is non-zero.
\end{proposition}
\begin{proof}
The assertion for $s=3$ has been observed in the first sentence of Example~\ref{sss112e}. The case $s\geq4$ then follows from (the proof of) Corollary~\ref{monotonicidad}.
\end{proof}

As anticipated in Example~\ref{sss112e}, we also describe, for $s\geq3$ and $e\geq2$, an almost sharp estimate for $\TC_s(F(1,1,2^e))$.

\begin{theorem}\label{2alaehigher} 
For $e\geq2$ and $s\geq3$,
$$
0\neq\left(z_{2,1}^{2^{e+1}-1} z_{2,2}^{2^{e+1}-1}\right) \cdot \left( z_{3,1}^{2^e+1} z_{3,2}^{2^e+3}\rule{0mm}{4mm}\right) \cdot\left( z_{4,1}^{2^e+1} z_{4,2}^{2^e}\rule{0mm}{4mm}\right)\cdots \left(z_{s,1}^{2^e+1} z_{s,2}^{2^e}\rule{0mm}{4mm}\right)\in H^*(F(1,1,2^e))^{\otimes s},
$$
consequently,  $\,s(2^{e+1}+1)-1\leq\TC_s(F(1,1,2^e))\leq s(2^{e+1}+1).$
\end{theorem}

\begin{remark}\label{creencia}{\em
Just as observed in Remark~\ref{dosnotasaclaras} in the case of Propositions~\ref{prodsnocerok12} and~\ref{prodsnocerok3}, the powers of the factors $z_{2,j}$ in the product element of Theorem~\ref{2alaehigher} coincide with the relevant powers of the product in~(\ref{elnocero1}) for $k=2$. The authors believe that such a phenomenon should shed light on possible generalizations of Theorems~\ref{familiasbis} and~\ref{2alaehigher}---see for instance Examples~\ref{exafins}.}\end{remark}

\begin{remark}{\em
Theorem~\ref{2alaehigher} fails for $e=1$ as $0=x_{1,2}^4+x_{3,2}^4=z_{3,2}^4\in H^*(F(1,1,2))^{\otimes s}$ in view of~(\ref{altuhigh}).
}\end{remark}

\begin{proof}[Proof of Theorem~\ref{2alaehigher}]
As in previous proofs, we can safely assume $s=3$. Further, although we should not focus now on the top dimensional basis element~(\ref{onlybasiselement}), the needed verifications are similar to those in the proof of Proposition~\ref{prodsnocerok3}. Indeed, this time we indicate how, for $e\geq2$, the basis element $x_{1,1}^{2^e}x_{1,2}^{2^e}\cdot x_{2,1}^{2^e+1}x_{2,2}^{2^e}\cdot x_{3,1}^{2^e+1}x_{3,2}^{2^e}$ appears in the expression of $(z_{2,1}^{2^{e+1}-1} z_{2,2}^{2^{e+1}-1}) \cdot ( z_{3,1}^{2^e+1} z_{3,2}^{2^e+3}\rule{0mm}{4mm})$ in terms of the tensor basis~(\ref{basetensorial}). The hypothesis $e\geq2$ is used for the analysis of the mod-2 arithmetics of binomial coefficients. That being said, the calculation details can easily be carried out by the diligent reader. As a guide, we note that the three key steps are
\begin{eqnarray*}
z_{3,1}^{2^e+1}z_{3,2}^{2^e+3}&\equiv&x_{1,1}x_{1,2}^2+x_{1,2}^3\,,\\z_{2,1}^{2^{e+1}-1}z_{2,2}^{2^{e+1}-1}&\equiv&x_{1,1}^{2^e-1}x_{1,2}^{2^e-2}+x_{1,1}^{2^e-2}x_{1,2}
^{2^e-1},
\end{eqnarray*}
and the easy fact that $(x_{1,1}x_{1,2}^2+x_{1,2}^3)\cdot(x_{1,1}^{2^e-1}x_{1,2}^{2^e-2}+x_{1,1}^{2^e-2}x_{1,2}
^{2^e-1})=x_{1,1}^{2^e}x_{1,2}^{2^e}$.
\end{proof}

\begin{remark}{\em\label{conclusion}
The results in this section suggest that purely cohomological methods can be used to give, for positive integers $i$ and $k$, an estimate of the higher topological complexity of $F(1^k,2^e-k+i)$ giving $G(k,2^e-k+i,s)<i$ provided $e$ is sufficiently large. An interesting additional restriction of the form
\begin{equation}\label{addres}
k+i-1\leq s,
\end{equation}
which would be compatible with the corresponding restrictions in Theorems~\ref{familiasbis} and~\ref{2alaehigher}, might also be needed. See Examples~\ref{exafins} below.
}\end{remark}

We close the section with a few examples offering small additional evidence to the result suggested in Remark~\ref{conclusion}.

\begin{examples}\label{exafins}{\em
For $F_5$ (i.e.~the flag manifold $F(1^k,2^e-k+i)$ in Remark~\ref{conclusion} with $k=4$, $e=2$, and $i=1$), a computer can be used to verify the non-triviality of $$z_{2,1}^7z_{2,2}^6z_{2,3}^3z_{2,4}^2\cdot z_{3,1}^1z_{3,2}^3z_{3,3}^5z_{3,4}^3\in H^*(F_5)^{\otimes 3}.$$ (The use of the sequence of exponents $7,6,3,2$ follows Remark~\ref{creencia} and the evidence noted right after~(\ref{evidentia}); the sequence of exponents $1,3,5,3$ was found---together with other five such sequences---by exhaustive computer calculations). This implies $\TC_s(F_5)=s\dim(F_5)=10 s$ for $s\geq3$. However, the computer experimentation also reports that, for $F(1^4,5)$ (i.e.~the flag manifold $F(1^k,2^e-k+i)$ now with $k=4$, $e=3$, and $i=1$), there is no corresponding non-zero product of the form $$z_{2,1}^{15}z_{2,2}^{14}z_{2,3}^7z_{2,4}^6\cdot z_{3,1}^az_{3,2}^bz_{3,3}^cz_{3,4}^d\in H^*(F(1^4,5))^{\otimes 3}$$ with $a+b+c+d=36$---which would yield the hoped-for equality
\begin{equation}\label{hopedfor}
\TC_s(F(1^4,5))=s\dim(F(1^4,5))=26\hspace{.2mm}s
\end{equation}
for $s\geq3$. Such an apparently unsuccessful situation is in fact compatible with (and reinforces)~(\ref{addres}), specially because a major computer search reports that
\begin{equation}\label{busqueda}
z_{2,1}^{15}z_{2,2}^{14}z_{2,3}^7z_{2,4}^6\cdot z_{3,1}^7z_{3,2}^7z_{3,3}^7z_{3,4}^{14}\cdot z_{4,1}^5z_{4,2}^7z_{4,3}^7z_{4,4}^8\in H^*(F(1^4,5))^{\otimes 4}
\end{equation}
(together with other 95 such polynomials in degree 104)
is non-zero. This of course implies that ~(\ref{hopedfor}) does hold as long as $s\geq 4$.
}\end{examples}

\section{Surfaces}\label{seccionsuperficies}
While Farber's topological complexity of a closed orientable surface was computed in the early work~\cite{MR1957228}, the non-orientable case has stood as a particularly intriguing task. For the lowest genus, it is known that $\TC(\mathbb{R}\mathrm{P}^2)=\Imm(\mathbb{R}\mathrm{P}^2)=3$, whereas a sophisticated obstruction theoretic analysis is given in~\cite{drani} to prove that (just as for oriented surfaces of genus at least 2) the topological complexity of a closed non-orientable surface of genus at least $4$ is 4. The topological complexity of the non-orientable surfaces of genera 2 and 3 remain still undecided. The goal of this section is to show that, just as for the semi complete flag manifolds in the previous section, the higher topological complexity of all (orientable or not) closed surfaces is fully accessible from cohomological methods.

\begin{proposition}\label{teosurbis}
Let $\Sigma$ stand for a closed surface (orientable or not) other than the sphere or the torus. Then $\TC_s(\Sigma)=2s$ for all $s\geq3$. 
\end{proposition}

The 2-sphere $S^2$ and the torus $T$ are true exceptional cases, as they have $\TC_s(S^2)=s$ and $\TC_s(T)=2s-2$ for any $s\geq2$ (see~\cite{MR3331610}).

\begin{proof}
Proof details for the oriented case are freely accessible from~\cite{morfismos}, and we only consider here the non-oriented case. For a positive integer $n$, let $N_n=\displaystyle{\sharp_{i=1}^n \mathbb{R}\mathrm{P}^2}$ denote the closed non-orientable surface of genus $n$. Recall that the mod 2 cohomology ring $H^*(N_n)$ is generated by $1$-dimensional classes $a_1,\ldots,a_n$ subject only to the relations $a_i^2=a_j^2$ and $a_ia_j=0$ for all $1\leq i,j\leq n$ with $i\neq j$, together with the relation $a_1^3=0$ when $n=1$ (the latter relation is obviously superfluous if $n>1$). In particular, there is a ring monomorphism $\varphi\colon H^*(N_1)\hookrightarrow H^*(N_n)$ determined by $\varphi(a_1)=a_1$. The corresponding ring monomorphism $\varphi_s:=\varphi^{\otimes s}\colon H^*(N_1)^{\otimes s}\hookrightarrow H^*(N_n)^{\otimes s}$ satisfies $\varphi_s(z_{i,1})=c_{i,1}$ for $2\leq i\leq s$, where the classes $z_{i,1}$ are defined in the paragraph following  Remark~\ref{aclaracion} (note that $N_1=\mathbb{R}\mathrm{P}^2=F(1,2)$), and the $s$-th zero-divisors $c_{i,1}$ are given by
$$
c_{i,1}=\underbrace{a_1 \otimes 1 \otimes \cdots \otimes 1}_{s\,\, \text{factors}} +\underbrace{1 \otimes \cdots \otimes 1 \otimes \overset{i}{a_1} \otimes 1 \otimes \cdots \otimes 1}_{s\,\, \text{factors}} \in H^*(N_n)^{\otimes s}.
$$
Here an $i$ on top of a tensor factor indicates the coordinate where the factor appears. The equality $\TC_s(N_n)=2s$ (for $s\geq 3$) now follow from Proposition~\ref{cotasseq} since  
$$
c_{2,1}^3c_{3,1}^3\prod_{i=4}^{s} (c_{i,1})^2\neq 0
$$
in view of Proposition~\ref{prodsnocerok12}(\ref{paraadelante}).
\end{proof}


\end{document}